\documentclass[]{interact}

\usepackage{epstopdf}
\usepackage[caption=false]{subfig}

\usepackage[longnamesfirst,sort]{natbib}
\bibpunct[, ]{(}{)}{;}{a}{,}{,}
\renewcommand\bibfont{\fontsize{10}{12}\selectfont}

\theoremstyle{plain}
\newtheorem{theorem}{Theorem}[section]
\newtheorem{lemma}[theorem]{Lemma}

\theoremstyle{definition}
\newtheorem{definition}[theorem]{Definition}
\newtheorem{example}[theorem]{Example}

\theoremstyle{remark}

\begin{document}


\title{A solution for fractional PDE constrained optimization problems using reduced basis method}

\author{
\name{A. Rezazadeh\textsuperscript{a} and M. Mahmoudi\textsuperscript{a} and M. Darehmiraki\textsuperscript{b}\thanks{Email: darehmiraki@gmail.com}   }
\affil{\textsuperscript{a}Department of Mathematics, University of Qom, Qom 37161466711, Iran;\\ \textsuperscript{b}Department of Mathematics, Behbahan Khatam Alanbia University of Technology, Behbahan, Khouzestan, Iran}
}

\maketitle

\begin{abstract}
In this paper, we employ a reduced basis method for solving the PDE constrained optimization problem governed by a fractional parabolic equation with the fractional derivative in time from order $ \beta\in (0,1) $ is defined by Caputo fractional derivative. Here we use optimize-then-discretize method to solve it. In order to this, First, we extract the optimality conditions for the problem, and then solve them by reduced basis method.
To get a numerical technique, the time variable is discretized using a finite difference plan. In order to show the effectiveness and accuracy of this method, some test problem are considered, and it is shown that the obtained results are in very good agreement with exact solution.
\end{abstract}

\begin{amscode}
Primary 43A62; Secondary 42C15.
\end{amscode}

\section{Introduction}
Fractional calculus appear in many settings across engineering and science disciplines, it can be refered to \cite{Sun} for more information. Also, many real-life examples are modeled by PDE constarined optimization, such as the optimal control of tumor invasion, the optimal strategy of a thermal treatment in cancer therapy, and the medical image analysis \cite{Mang,Qui,Qui2}. Thanks to the increasing use of fractional derivative and fractional calculus in ordinary and partial differential equations and related problems, there is a interest for presenting efficient and reliable solutions for them.   \\
The reduced basis method is a model order reduction effective technique for approximating the solution of parameterized partial differential equations which have been used in the past few decades. Firstly, the RB method was proposed to solve the nonlinear analysis of structures in the late 1970s and has subsequently been further investigated and developed more broadly \cite{Balmes,Grepl}. The main idea of the RB method for parameter dependent PDEs is to approximate its solution by a linear combination of few global basis functions, obtained from a set of FE solutions corresponding to different parameter values. The RB method was repeatedly used to solve the PDE constrained optimization problem (PCOP). A reduced basis surrogate model was proposed for PCOP in \cite{r1}. Authors in \cite{r2} used RB method in conjunction with a trust region optimization framework to accelerate PCOP. RB method was applied to solve quadratic multi-objective optimal control problems governed by linear parameterized variational equations in \cite{r3}. Authors in \cite{r4} employed a locally adapted RB method for solving risk-averse PCOP.
\subsection{Literature review}
Many researchers pay attention to finding the numerical solution  for solving fractional PDE constrained optimization problem. 
\cite{Antil}
 use finite element method (FEM) for an Optimal Control Problem of Fractional Powers of Elliptic Operators. In \cite{Ota} a piecewise linear FEM was proposed for an optimal control problem involving the fractional powers of a symmetric and uniformly elliptic second order operator. A spectral method for optimal control problems governed by the time fractional diffusion equation with control constraints was presented in \cite{Ye}. Authors in \cite{Bhrawy2} provided a space-time Legendre spectral tau method for the two-sided space-time Caputo fractional diffusion-wave equation. A new spectral collocation algorithm for solving time-space fractional partial differential equations with subdiffusion was reported in \cite{Bhrawy}. A hybrid meshless method was proposed for FODCP in \cite{Dareh}. Zacky and Mchado \cite{Zaky} provided a solution for FODCP by pseudo-spectral method.  As well as there exist variety of solutions for various fractional optimal control problems that the some of them can be seen in \cite{Salati,Bai}.\\

\subsection{The main aim of this paper}
Our main motivation in this paper is to apply reduced basis method for the rapid and reliable solution of fractional PDE constrained optimization problem. There exists, to the best of our knowledge, no previous work providing reduced order approximations for fractional PDE constrained optimization problem. It is known that the reduced basis method has been an efficient tool for computing numerical solutions of differential equations because of its high-order accuracy. We demonstrate the efficiency of the proposed schemes by considering several numerical examples.\\
The outline of this paper is as follows. After stating the exact form of the problem in Section 2, we present the weak form of the problem in Section 3. Necessary and sufficient conditions for the problem are introduced in section 4. Sections 5 review finite element method and reduced basis method. In section 6, we give some implementation details and present the numerical results to support the theoretical prediction. The error bound for proposed method is calculated in Section 7. Three numerical examples are presented in Section 8 to confirm our theoretical findings. In Section 9, we give some concluding remarks.
\section{The exact form of the problem}\label{se2}
This paper solves an optimal control problem with fractional PDE constraint. In this section, we present the parameterized optimal control problem. \\

Let $N\geqslant 1$ and $\Omega$ be a bounded open subset of $\mathbb{R}^N$ with boundary $\partial \Omega$. For a time $T>0$, we set $Q=\Omega \times (0,T)$ and $\Sigma=\partial \Omega \times (0,T) $. Given $\alpha \in (0,1)$,  and $D\subset \mathbb{R}^p$ be a $p$-dimensional compact set of parameters $\mu=(\mu_1,...\mu_p)$ with $p\geq 1$. Let $Y$ with $H^{1}_{0}(\Omega)\subset Y \subset H^1(\Omega)$ be a Hilbert space with inner product $(.,.)_Y$, $(v,w)=\int_{\Omega} vw d\Omega$
and associated norm $\|.\|_Y=\sqrt{(.,.)_Y}$. We assume that the norm $\|.\|_Y$ is equivalent to $H^1(\Omega)$-norm and denote the dual space of $Y$ by $Y^{'}$. We also introduce the control Hilbert space $U$ with its inner product $(.,.)_{U}$ and norm $\|.\|=\sqrt{(.,.)_{U}}$ with associated dual space $U^{'}$.
 We define the parameterized  fractional optimal control problem:
 \begin{align}\label{eq1.1}
\min \quad J(y,u)=\frac{1}{2}\int_{0}^{T}\int_{\Omega}&(y(.;\mu)-y_d)^2  d\Omega dt +\frac{\gamma}{2}\int_{0}^{T}\int_{\Omega}u^2(.;\mu) d\Omega dt, \nonumber\\
\text{subject to}\nonumber\\
 -\partial^{\alpha}_ty+\mu\Delta y&= u,\qquad \qquad \qquad\qquad\quad\text{ in $Q$},   \\
 y&=0, \quad \quad \quad \quad\quad \quad\quad\qquad\, \text{on $\Sigma$},\nonumber \\
y(.,0)&= 0, \quad \quad \quad \qquad\qquad \qquad \text{ in $\Omega$},\nonumber
\end{align}
where $J(y,u)$ is the cost functional, $\gamma >0$ is the so-called regularization parameter and $y_d$ is the desired state. Here the variables $y\in H^1(Q)$ and $ u\in L_2(Q)$ are named the state variable and the control variable, respectively. They are independent variables and should be determined.\\
The fractional derivative $\partial^{\alpha}_t$ for $\alpha\in (0,1)$ is the left-sided Caputo fractional derivative of order $\alpha$ with respect to $t$ and defined as:
\begin{align}\label{eq1.2}
\partial^{\alpha}_tf(t)=\frac{1}{\Gamma(1-\alpha)}\int_0^ t \frac{f^{'}(\tau)}{(t-\tau)^{\alpha}}d\tau,
\end{align}
where $\Gamma$ is the Gamma function. We consider $\partial_t$ for $\alpha=1$. The right-sided Caputo fractional is
\begin{align}\label{eq1.3}
\partial^{\alpha}_{T-t}f(t)=-\frac{1}{\Gamma(1-\alpha)}\int_t^ T \frac{f^{'}(\tau)}{(\tau -t)^{\alpha}}d\tau \quad\quad\quad \alpha\in (0,1),
\end{align}
where $f\in L^1(0,T)$.\\

\begin{definition}\label{de5.1}(Steeb, 1997; Steeb, 2006) Suppose $C=(c_{ij})_{m\times n}$ and $D$ are two arbitary matrices, then the matrix
\begin{align*}
C\otimes D=\begin{bmatrix}
c_{11}D& c_{12}D&\cdots & c_{1n}D\\c_{21}D&c_{22}D&\cdots &c_{2n}D\\
\vdots &\vdots \ddots &\vdots\\ c_{m1}D&c_{m2}D&\cdots &c_{mn}D
\end{bmatrix}
\end{align*}
is named Kronocker product of $C$ and $D$.
\end{definition}
\begin{definition}\label{de5.2} Suppose $C=(c_{ij})_{m\times n}$ be a given matrix, then $vec(C)$ is  a column vector  made of the row of $C$ stacked a top one another from left to right that the size of it is  $m\times n$.
\[vec(C)=(c_{11},c_{12},...,c_{1n},c_{21},c_{22},...,c_{m1},...,c_{mn})^T.\]
\end{definition}
\section{Weak form of the problem}\label{se3}
In this section, we consider the weak form of the problem \eqref{eq1.1}.
 \begin{align}\label{eq2.1}
\min \quad J(y,u)=\frac{1}{2}\int_{0}^{T}\int_{\Omega}&(y(.;\mu)-y_d)^2  d\Omega dt +\frac{\gamma}{2}\int_{0}^{T}\int_{\Omega}u^2(.;\mu) d\Omega dt, \nonumber\\
&\text{S.t}\nonumber\\
 -\partial^{\alpha}_t(y(t),v)-a(y(t),v;\mu)&=(u(t),v),\qquad \qquad\quad  v\in Y,   \\
 y&=0, \qquad \qquad\qquad \qquad \text{on $\partial \Omega$},\nonumber \\
 y(.,0)&= 0,  \qquad \qquad\qquad\qquad\, \text{ in $\Omega$}.\nonumber
\end{align}
 We introduce the parameter-dependent bilinear form $a(.,.;\mu):Y\times Y\rightarrow \mathbb{R}$ as
  \[a(v,w;\mu)=\mu\int_{\Omega}\nabla v\nabla w d\Omega.\] It is assumed that $a(.,.;\mu)$ is continious
\[0\leq \gamma_a(\mu)=sup_{w\in Y/\{0\}}sup_{v\in Y/\{0\}} \frac{a(w,v;\mu)}{\|w\|_Y \|v\|_Y}\leq \gamma^a_0< \infty,    \quad\quad \forall \mu\in D,\]
and coercive
\begin{align}\label{eq2.2}
\alpha(\mu)=inf_{v\in Y/\{0\}}\frac{a(v,v;\mu)}{\|v\|^2_Y}\geq \alpha_0>0, \quad\quad\quad\quad\quad\quad\quad\quad\quad\quad \forall \mu\in D.
\end{align}
The bilinear form is assumend to be dependent affinely on the parameter.
\section{Optimality condition}\label{se4}
Now, the first order necessary and sufficient optimality conditions for the fractional optimal control problem \eqref{eq2.1} are derived using the Lagrangian approach. Let $p$ denote the adjoint variable, $\mu\in D$ is given, we define the Lagrangian functional related to Problem \eqref{eq2.1} as
\begin{align*}
L(y,u,p)=J(y,u)-\int_Q(-\partial^{\alpha}_t(y(t),v)-a(y(t),v;\mu)-(u(t),v))pdx dt.
\end{align*}
By derivative of $L$ with respect to $p$, the state equation is as following:
\begin{align}\label{eq2.2}
\begin{cases}
&-\partial_{t}^{\alpha}(y^*(t),\phi)-a(y^*(t),\phi;\mu)=(u^*,\phi), \quad \quad\quad  \phi \in Y,\\
&y^*(.,t)=0, \qquad \qquad\qquad\qquad\qquad\qquad\qquad\quad\text{on $\partial \Omega $},\\
 &y^*(.,0)=0  \qquad \qquad\qquad\qquad\qquad\qquad\qquad\quad\,\,\text{in $\Omega $}.
 \end{cases}
 \end{align}
By differentiating of $L$ with respect to $y$,  the adjoint equations are obtained
\begin{align}\label{eq2.3}
 \begin{cases}
&-\partial_{T-t}^{\alpha}(p^*(t),\varphi) -a(p^*(t),\varphi;\mu)=(y_d(t)-y^*,\varphi), \quad  \varphi \in Y,\\
&p^*(.,t)=0,  \qquad \qquad\qquad\qquad\qquad\qquad\qquad\qquad\quad\,\text{on $\partial \Omega $},\\
&p^*(.,T)=0, \qquad \qquad\qquad\qquad\qquad\qquad\qquad\quad\,\,\,\,\,\,\quad\text{in $\Omega $}.
\end{cases}
 \end{align}
 Finally, by derivative with respect to $u$, the gradient equation is obtained
 \begin{align}\label{eq2.4}
\gamma u^*-p^*=0,\quad\quad\quad\quad \text{in $Q$}.
\end{align}
We replace $u^*=\frac{1}{\gamma}p^*$ in equation \eqref{eq2.2}. Since the fractional state equation is left-sided Caputo fractional derivative and the fractional adjoint equation is right-sided, we apply change of variable $p(x,t)=\bar{p}(x,T-t)$ and then we have
\[\partial_{T-t}^{\alpha}p(x,t)=\partial_{t}^{\alpha}\bar{p}(x,T-t),\] for more detail, see \citep{Antil}. As a result, the right Caputo fractional derivative can be written as a left Caputo fractional derivative. Therefore, the equation \eqref{eq2.3}, \eqref{eq2.4} change to
\begin{align}\label{eq2.5}
 \begin{cases}
 &-\partial_{t}^{\alpha}(y^*(t),\phi)-a(y^*(t),\phi;\mu)=\frac{1}{\gamma}(\bar{p}^*(T-t),\phi), \quad \quad\quad \quad\quad\,\, \phi \in Y,\\
&-\partial_{t}^{\alpha}(\bar{p}^*(T-t),\varphi) -a(\bar{p}^*(T-t),\varphi;\mu)=(y_d-y^*,\varphi),  \quad\quad\quad \varphi \in Y,\\
&y^*(.,t)=0, \quad \bar{p}^*(.,T-t)=0, \qquad\qquad\qquad\qquad\qquad \qquad\quad\text{on $\partial \Omega $},\\
&y^*(.,0)=0, \quad \bar{p}^*(.,0)=0, \qquad \qquad\qquad \qquad\qquad\qquad \qquad\quad\text{in $\Omega. $}
\end{cases}
 \end{align}
For simplicity, we hereinafter eliminate the star icon in the upper index.
 \section{Finite element method}\label{se5}
 In this section, finite element method is used to discretize the problem. We divide the time interval $[0,T]$ into $K$ sub-interval of equal length $\Delta t=\frac{T}{K}$ and $t^k=k\Delta t, 0\leq k\leq K$ and  $\mathbb{K}=\{0,1,..., K\}$. We introduce the finite element spaces $Y_h\subseteq Y$ and $P_h\subseteq P$ of large dimension $N_h$. Then the following equations are gained
 \begin{align}\label{eq2.5}
 \begin{cases}
 &-\partial_{t}^{\alpha}(y(t_k),\phi)-a(y(t_k),\phi;\mu)=\frac{1}{\gamma}(\bar{p}(T-t_k),\phi), \quad \quad\quad \quad\quad \quad\quad\phi \in Y_h, k\in \frac{ \mathbb{K}}{\{0\}},\\
&-\partial_{t}^{\alpha}(\bar{p}(T-t_k),\varphi) -a(\bar{p}(T-t_k),\varphi;\mu)=(y_d(t_k)-y(t_k),\varphi),  \quad\quad \varphi \in Y_h, k\in \frac{ \mathbb{K}}{\{K\}},\\
&y(.,t_k)=0, \quad \bar{p}(.,T-t_k)=0, \quad\quad \quad\quad\quad\qquad \qquad\qquad\qquad\qquad\text{on $\partial \Omega $},\\
&y(.,0)=0, \quad \bar{p}(.,T-t_{K})=0, \quad \quad\quad\quad\quad \qquad\qquad\qquad\qquad\qquad\text{in $\Omega $}.
\end{cases}
 \end{align}
 Then, in algebraic formulation we have,
 \begin{align}\label{eq2.5}
 \begin{cases}
 &- M\partial_{t}^{\alpha}y(t_k)-\mu Ay(t_k)-\frac{1}{\gamma}B\bar{p}(T-t_k)=0, \quad \quad\quad\quad\quad\quad \quad\quad k\in \frac{ \mathbb{K}}{\{0\}},\\
&-M^{'}\partial_{t}^{\alpha}(\bar{p}(T-t_k) -\mu A^{'}\bar{p}(T-t_k)+B^{'}y(t_k)=Y_d(t_k), \quad \quad\quad  k\in \frac{ \mathbb{K}}{\{K\}},\\
&y(.,t_k)=0, \quad \bar{p}(.,T-t_k)=0, \quad\quad \quad\quad\quad\qquad \qquad\qquad\qquad\quad\text{on $\partial \Omega $},\\
&y(.,0)=0, \quad \bar{p}(.,T-t_K)=0, \quad\quad \quad\quad\quad\qquad \qquad\qquad\qquad\quad\,\,\,\text{in $\Omega $},
\end{cases}
\end{align}
where
\begin{align*}
&y_h(t_k)=\sum_{i=1}^{N_h}y_{hi}(t_k)\phi_i ,\quad\quad\quad\quad\quad \bar{p}_h(T-t_k)=\sum_{i=1}^{N_h}\bar{p}_{hi}(T-t_k)\varphi_i ,
\end{align*}
$\phi_i$ and $\varphi_i$ are the basis of finite element method.
The matrices used in the  state and adjoint equations are
\begin{align*}
&(M)_{ij}=(\phi_i,\phi_j),\quad\quad\quad\quad\quad\quad\quad (B)_{ij}=(\varphi_i,\phi_j),\\
&(M^{'})_{ij}=(\varphi_i,\varphi_j), \quad\quad\quad\quad\quad\quad\quad (B^{'})_{ij}=(\phi_i,\varphi_j),\\& Y_d(t_k)=(y_d(t_k),\varphi_j),
\end{align*}
where $M, M^{'}, B$ and $B^{'}$ are the mass matrix of the finite element method. Also, we have
\[(A)_{ij}=a(\phi_i,\phi_j),\quad\quad (A^{'})_{ij}=a(\varphi_i,\varphi_j),\quad\quad\quad 1\leq i,j\leq N_h\]
where they are the same and are called the stiffness matrix of finite element method.\\
Now, we express the following lemma for discretization of time fractional derivative.
\begin{lemma}\label{Le5.1}\cite{Mohebbi}.
Suppose $0\leq\alpha \leq 1$ and $g(t)\in C^2[0,t_k]$, it holds that
\begin{align*}
|\frac{1}{\Gamma(1-\alpha)}\int_0^{t_n}\frac{g^{'}(t)}{(t_n-t)^{\alpha}}dt-c[b_0g(t_n)-\sum_{m=1}^{n-1}(b_{n-m-1}-b_{n-m})g(t_m)-b_{n-1}g(t_0)]|\leq\\
\frac{1}{\Gamma(2-\alpha)}[\frac{1-\alpha}{12}+\frac{2^{2-\alpha}}{2-\alpha}-(1+2^{-\alpha})]\max_{0\leq t\leq t_k}|g^{"}(t_n)|\tau^{2-\alpha},
\end{align*}
where \[b_m=(m+1)^{1-\alpha}-m^{1-\alpha} , \quad \quad c=\frac{\tau^{-\alpha}}{\Gamma(2-\alpha)},\quad\quad \tau=\Delta t .\]
\end{lemma}
Therefore, for $n=1$,
\[\partial_t^{\alpha}g(t_1)\simeq  c[g(t_1)-g(t_0)],\]
and for $2\leq n\leq L$,
\[\partial_t^{\alpha}g(t_n)\simeq c[b_0g(t_n)-\sum_{m=1}^{n-1}(b_{n-m-1}-b_{n-m})g(t_m)-b_{n-1}g(t_0)] .\]
Since $y(t_0)=0$ and $\bar{p}(0)=\bar{p}(T-t_K)=0$, we have
the following estimation
\[\partial_t^{\alpha}y(t_n)\simeq c[b_0y(t_n)-\sum_{m=1}^{n-1}(b_{n-m-1}-b_{n-m})g(t_m),\]
and \[\partial_t^{\alpha}\bar{p}(T-t_n)\simeq c[b_0\bar{p}(T-t_n)-\sum_{m=1}^{K-n-1}(b_{K-n-m-1}-b_{K-n-m})\bar{p}(T-t_{K-m}).\]
We define a matrix
\[D=\begin{bmatrix}
cb_0& 0&0&\cdots & 0\\c(b_1-b_0)&cb_0&0&\cdots &0\\c(b_2-b_1)& c(b_1-b_0)&cb_0&\cdots&0\\
\vdots &\vdots \ddots&\vdots&\vdots &\vdots\\ c(b_{K-1}-b_{K-2})&c(b_{K-2}-b_{K-3})&\cdots&c(b_1-b_0) &cb_0
\end{bmatrix}
\]
and
\[
vec(y)=[y_1(t_1),\cdots,y_{N_h}(t_1),y_1(t_2),\cdots,y_{N_h}(t_2),\cdots, y_1(t_{K}),\dots,y_{N_h}(t_{K})],\]
\[vec(y_d)=[y_{d1}(t_0),\cdots,y_{dN_h}(t_0),y_{d1}(t_1),\cdots,y_{dN_h}(t_1),\cdots, y_{d1}(t_{K-1}),\dots,y_{dN_h}(t_{K-1})],\]
\begin{align*}
vec(\bar{p})=[&\bar{p}_1(T-t_0),\cdots,\bar{p}_{N_h}(T-t_0),\bar{p}_1(T-t_1),\cdots,\bar{p}_{N_h}(T-t_1),\cdots, \\ &\bar{p}_1(T-t_{K-1}),\cdots,\bar{p}_{N_h}(T-t_{K-1})].
\end{align*}
By using Lemma \ref{Le5.1} the algebraic form of equation \eqref{eq2.5} is as following:
\begin{align}
\begin{cases}
-[D\otimes M+\mu I_{K}\otimes A]vec(y)-\frac{1}{\gamma}[I_b\otimes M ]vec(\bar{p})=0,\\
\\
[I_b^{T}\otimes M]vec(y)-[D^{T}\otimes M+\mu I_{K}\otimes A]vec(\bar{p})=vec(Y_d),
\end{cases}
\end{align}
where $I_{K}$ is the $K\times K$ unitary matrix and $I_b $ is a $K\times K$ matrix which the first diag of it is equal to one.
\section{Reduced basis method}\label{se6}
In general, the RBM constructs the reduced basis using the greedy algorithm and pre-compute the parameter independent parts of matrices at the off-line stage. We assemble the matrices using the coefficients at new parameter, solve the system and compute the output at the on-line stage. In the whole process, we restrict the approximate space to the much smaller subspace chosen by the greedy algorithm and discard the unnecessary modes during the calculation of the basis.
\subsection{Off-line stage}
We employ the reduced basis method for the efficient solutions of equation  \eqref{eq2.5}.  We first assume that a sample set
 $S_N=\{\mu_1,...,\mu_N\}$ is given, the reduced basis spaces are
\[Y_N=span\{\zeta^y_n, 1\leq n \leq N \},\]
\[P_N=span\{\xi^p_n, 1\leq n \leq N\}.\]
Here $\zeta^y_n$ and $\xi^p_n$, $1\leq n\leq N$ are $(.,.)_Y$-orthogonal basis functions derived by a Gram-Schmidt orthogonalization procedure and also $N\ll N_h$. We remark the greedy sampling procedure to construct $S_N, Y_N, P_N$ in Section 5.2.
The reduced  optimal control problem is as following:
\begin{align}\label{eq6.1}
\min \quad J(y_N,u_N;\mu)=\frac{1}{2}\int_{Q}\|y_N(.;\mu)-y_{Nd}&\|^2_{Y_N}+\frac{\gamma}{2}\int_{Q}\|u_N(.;\mu)\|^2_{U_N},\\
\text{S.t}\nonumber\\
 -\partial _t^{\alpha}(y_N,v)- a(y_N,v;\mu)&=(u_N,v),\qquad  \qquad\qquad\,\,\quad\, v\in Y_N, \nonumber\\
y_N(.,t)&=0, \quad \quad\quad\quad \qquad\qquad \qquad\,\text{on $\partial \Omega$},\nonumber \\
y(.,0)&= 0,  \quad \quad\quad\quad \quad\quad\quad\qquad\quad\text{ in $\Omega$}, \nonumber
\end{align}
where $y_{Nd} $ is the projection of $y_d$ on the N dimensional space $Y_N$.\\
The first-order optimality system of \eqref{eq6.1} reads: Given $\mu\in D$, the optimal solution $(y_N,\bar{p}_N)$ satisfies\\
\begin{align*}
\begin{cases}
-\partial _t^{\alpha}(y_N(t),\phi)- a(y_N(t),\phi;\mu)=\frac{1}{\gamma}(\bar{p}_N(T-t),\phi),\quad \quad \quad \quad\quad \quad\quad \quad\quad\phi\in Y_N, \\
 \\
 -\partial _t^{\alpha}(\bar{p}_N(T-t),\varphi)- a(\bar{p}_N(T-t),\varphi;\mu)+(y_N(t),\varphi)=(y_d(t),\varphi),\quad \quad \varphi \in Y_N,\\
 \\
 y_N(.,0)= 0,\quad \quad \bar{p}_N(.,T-t_K)=0,\qquad \qquad\qquad \qquad\qquad\qquad\qquad\quad \text{in $\Omega$}, \\
\\
 y_N(.,t)=0 ,\quad \quad  \bar{p}_N(.,T-t)= 0,\qquad \qquad \qquad\qquad \qquad\qquad\quad\qquad\quad \text{on $\partial \Omega$}.
\end{cases}
\end{align*}
\\
The system for $t=t_k$ is as following:
 \begin{align}
 \begin{cases}\label{eq6.3}
 -\partial _t^{\alpha}(y_N(t_k),\phi)- a(y_N(t_k),\phi;\mu)=\frac{1}{\gamma}(\bar{p}_N(T-t_k),\phi),\quad \quad \quad\quad \quad\quad\quad \quad \phi\in Y_N, k \in\frac{\mathbb{K}}{\{0\}}, \\
 \\
  -\partial _t^{\alpha}(\bar{p}_N(T-t_k),\varphi)- a(\bar{p}_N(T-t_k),\varphi;\mu)+(y_N(t_k),\varphi)=(y_d(t_k),\varphi), \quad \varphi \in Y_N, k \in\frac{\mathbb{K}}{\{K\}},\\
  \\
 y_N(.,0)= 0,\quad \quad  \bar{p}_N(.,T-t_{K})= 0,\quad \quad\qquad\qquad\qquad\qquad\qquad\qquad\qquad \text{on $\partial \Omega$},\\
\\ 
y_N(.,t_k)=0,\quad \quad \bar{p}_N(.,T-t_k)=0,\quad \quad\qquad\qquad\qquad\qquad\qquad\qquad \qquad\text{in  $ \Omega$}.
\end{cases}
\end{align}
The reduced state and adjoint solution are defined as
\[y_N(t_k)=\sum_{i=1}^Ny_{Ni}(t_k)\zeta_i ,\quad\quad \bar{p}_N(T-t_k)=\sum_{i=1}^N \bar{p}_{Ni}(T-t_k)\xi_i.\]
Denote the coefficient vectors as $y^k_N(t_k)=[y_{N1}(t_k),...,y^k_{NN}(t_k)]^T\in \mathbb{R}^N$  and $\bar{p}_N(T-t_k)=[\bar{p}_{N1}(T-t_k),...,\bar{p}_{NN}(T-t_k)]^T\in \mathbb{R}^N$, respectively.
By using Lemma \ref{Le5.1} the algebraic form of equation \eqref{eq6.3} is as following:
\begin{align}\label{eq6.4}
\begin{cases}
-[D\otimes M_N+\mu I_{K}\otimes A_N]vec(y_N)-\frac{1}{\gamma}[I_b\otimes B_N]vec(\bar{p}_N)=0,\\
\\
[I_b^{T}\otimes B^{'}_N]vec(y_N)-[D^{T}\otimes M^{'}_N+\mu I_{K}\otimes A^{'}_N]vec(\bar{p}_N)=vec(Y_{dN}).
\end{cases}
\end{align}
Here, $M_N, M^{'}_N, B_N,B^{'}_N\in R^{N\times N}$ are matrices with entries
\[ (M_N)_{ij}=(\zeta_i,\zeta_j), \quad\quad\quad\quad (M^{'}_N)_{ij}=(\xi_i,\xi_j),\]
\[(B_N)_{ij}=(\zeta_i,\xi_j), \quad\quad\quad\quad (B^{'}_N)_{ij}=(\xi_i,\zeta_j). \] The parameter dependent  matrices
$A_N, A^{'}_N\in R^{N\times N}$ are given by

\[ (A_N)_{ij}=a(\zeta_i,\zeta_j), \quad\quad\quad\quad (A^{'}_N)_{i,j}=a(\xi_i,\xi_j)\]
  and $Y^k_{dN}$ is defined as :
\[Y_{dN}(t_k)=(y_d(t_k),\xi_i),\]
\[
vec(y_N)=[y_1(t_1),\cdots,y_{N}(t_1),y_1(t_2),\cdots,y_{N}(t_2),\cdots, y_1(t_{K}),\dots,y_{N}(t_{K})],\]
\[vec(y_{dN})=[y_{d1}(t_0),\cdots,y_{dN}(t_0),y_{d1}(t_1),\cdots,y_{dN}(t_1),\cdots, y_{d1}(t_{K-1}),\dots,y_{dN}(t_{K-1})],\]
\begin{align*}
vec(\bar{p}_N)=[&\bar{p}_1(T-t_0),\cdots,\bar{p}_{N}(T-t_0),\bar{p}_1(T-t_1),\cdots,\bar{p}_{N}(T-t_1),\cdots, \\ &\bar{p}_1(T-t_{K-1}),\cdots,\bar{p}_{N}(T-t_{K-1})].
\end{align*}
\subsection{ Greedy Sampling Procedure}
We generate the reduced basis space using the greedy sampling procedure summarized in Algorithm 1. $D_{\text{train}}\subset D$ is a finite but suitably large parameter train sample; $\mu_1$ is the initial parameter value and $\epsilon >0$ is a prescribed desired error tolerance. Since we can only guarantee the desired error tolerance for all $\mu\in D_{\text{train}}$, we note that we have to choose the train sample sufficiently fine. Now we present the greedy algorithm to generate the reduced basis spaces.
\\\subsubsection*{Algorithm 1}Greedy procedure\\
\begin{tabular}{l|l}
  \hline

  1 & Choose $D_{\text{train}}\subset D$, $\mu_1\in D_{\text{train}}$(arbitary), and $\epsilon >0$. \\\hline
  2 & N=1, $\mu^*\leftarrow\mu_1$, $S_N=\{\mu^*\}$, $\triangle_{N}(\mu^*)\leftarrow\triangle_{Ny}(\mu^*)+\triangle_{Np}(\mu^*)$, \\\hline
  3 & Set $Y_N=POD(\{y^1(\mu^*),...,y^{K}(\mu^*)\})$ and\\ &  $P_N=POD(\{p^0(\mu^*),...,p^{K-1}(\mu^*)\})$,
 \\\hline
  4 & while  $\triangle_{N}(\mu^*)>\epsilon$ do \\\hline
  5 & $\forall \mu \in D_{\text{train}}$ compute the reduced solution $(y_N,p_N)$  and\\ &  $\triangle_{Ny}(\mu)\leftarrow \|y_h(\mu)-y_N(\mu)\|$,\\ &    $\triangle_{Np}(\mu)\leftarrow \|p_h(\mu)-p_N(\mu)\|$, $\triangle_{N}(\mu)\leftarrow \|\triangle_{Ny}(\mu)\|+\|\triangle_{Np}(\mu)\|$, \\\hline
  6 & $\mu^*\leftarrow argmax_{\mu \in D_{\text{train}}}\triangle_{N}(\mu), S_N=S_N\cup \{\mu^*\}$, $Y_N=Y_N\cup POD(\{y^1(\mu^*),...,y^{K}(\mu^*)\})$,\\& $P_N=P_N\cup POD(\{p^0(\mu^*),...,p^{K-1}(\mu^*)\})$,\\\hline
  7 & $\triangle_{N}(\mu^*)\leftarrow max_{\mu \in D_{\text{train}}} \triangle_{N}(\mu)$, \\\hline
  8 & $N\longleftarrow N+1$, \\\hline
  9 & end while. \\
  \hline
\end{tabular}
We note that the reduced basis $Y_N$ and $P_N$ are enlarged in step 6 according to the evaluation of $\triangle _N(\mu)$ defined in step 5 as the real error between the FE solution and the RB one. This is  computationally expensive and time-consuming. To avoid this, we obtain the error bound described in Section 6.
\subsection{Online Stage}
Let $\mu\in D$ is given. By using the information obtained in the offline stage, we solve the system \eqref{eq6.4} for the  given parameter. For more detail about reduced basis method, see \cite{kar}
\section{Error bound}\label{se7}
To begin, we introduce the dual norm of the primal (state) residual 
\begin{align}\label{eq7.11}
\varepsilon^{pr}(\mu,t_k)\equiv \sup_{v\in Y}\frac{R^{pr}(v;\mu,t_k)}{\|v\|_Y},
\end{align}
and the dual norm of the dual (adjoint) residual
\begin{align}
\varepsilon ^{du}(\mu,T-t_k)\equiv \sup_{v\in Y}\frac{R^{du}(v;\mu,T-t_k)}{\|v\|_Y},
\end{align}
where 
\begin{align}
&R^{pr}(v;\mu,t_k)=(u^k,v)+a(y_N(\mu,t_k),v)-cb_{k-1}(y_N(t_0),v)\label{eq7.2}\\&-c\sum_{m=1}^{k-1}(b_{k-m-1}-b-{k-m})(y_N(\mu,t_m),v)+cb_0(y_N(\mu,t_k),v),\nonumber \\
&R^{du}(v;\mu,T-t_k)=(y_d(t_k),v)-(y(\mu,t^k),v)+a(\bar{p}_N(\mu,T-t_k),v)\label{eq7.3}\\&-c\sum_{m=1}^{K-k-1}(b_{K-k-m-1}-b_{K-k-m})(\bar{p}_N(\mu,T-t_{K-m}),v)+cb_0(\bar{p}_N(\mu,T-t_k),v).\nonumber
\end{align}
\begin{lemma}\label{Le7.1}
Let $e^{pr}(\mu,t_k)\equiv y_(\mu,t_k)-y_N(\mu,t_k)$ be the error in the primal variable and define the norm 
\[\|v(\mu,t_k)\|^{pr}=((v(\mu,t_k),v(\mu,t_k);\mu)+a(v(\mu,t_k),v(\mu,t_k);\mu))^{\frac{1}{2}}.\]
The error in the primal variable is then bounded by
\[\|e^{pr}(\mu,t_k)\|^{pr}\leq [ \frac{1}{\alpha(\mu)}(\varepsilon^{pr}(\mu,t_{k}))^2+\sum_{k^{'}=1}^{k-1}\Delta_{k^{'}}]^{\frac{1}{2}},\]
where 
\[\Delta_{k^{'}}=\sum_{k^{"}=1}^{k^{'}}\frac{1}{\alpha(\mu)}(\varepsilon^{pr}(\mu,t_{k^{"}}))^2.\]
\end{lemma}
\begin{proof}
We immediately derive from Equation \eqref{eq7.11} that $e^{pr}(\mu,t_k)$ satisfies 
\[a(e^{pr}(\mu,t_k),v)+cb_0(e^{pr}(\mu,t_k),v)=c\sum_{m=1}^{k-1}(b_{k-m-1}-b_{k-m})(e^{pr}(\mu,t_m),v)+R^{pr}(v;\mu,t_k).\]
Set $v=e^{pr}(\mu,t_k)$, by equation \eqref{eq2.2}, we obtain
\begin{align*}
&a(e^{pr}(\mu,t_k),e^{pr}(\mu,t_k))+cb_0(e^{pr}(\mu,t_k),e^{pr}(\mu,t_k) )\\ &\leq c\sum_{m=1}^{k-1}(b_{k-m-1}-b_{k-m})(e^{pr}(\mu,t_m),e^{pr}(\mu,t_k))+\varepsilon^{pr}(\mu,t_{k})\|e^{pr}(\mu,t_k)\|\\&\leq
c\sum_{m=1}^{k-1}(b_{k-m-1}-b_{k-m})(e^{pr}(\mu,t_m),e^{pr}(\mu,t_k))+\varepsilon^{pr}(\mu,t_{k})\|e^{pr}(\mu,t_k)\|\\&\leq c\sum_{m=1}^{k-1}\max(b_{k-m-1}-b_{k-m})(e^{pr}(\mu,t_m),e^{pr}(\mu,t_k))+\varepsilon^{pr}(\mu,t_{k})\|e^{pr}(\mu,t_k)\|.
\end{align*}
We know that
\begin{align*}
& b_0=1, \quad \max(b_{k-m-1}-b_{k-m})=b_0-b_1<1,\quad (a,b)\leq\|a\| \|b\|, 
\end{align*}
so we have
\begin{align*}
&(e^{pr}(\mu,t_k),e^{pr}(\mu,t_k)+c(e^{pr}(\mu,t_k),e^{pr}(\mu,t_k)\leq
\\&
\leq c\sum_{m=1}^{k-1}\|e^{pr}(\mu,t_m)\|\|e^{pr}(\mu,t_k)\|+\varepsilon^{pr}(\mu,t_{k})\|e^{pr}(\mu,t_k)\|.
\end{align*}
Since
\[\text{ and (for }c, d\in \mathbb{R}, \rho\in \mathbb{R}_{+}), \quad  2|a||b|\leq \frac{1}{\rho^2}a^2+\rho^2 b^2,\]
we apply it twice, 
\[a=\sqrt{c}\sum_{m=1}^{k-1}\|e^{pr}(\mu,t_m)\|,\quad b=\sqrt{c}\|e^{pr}(\mu,t_k)\| ,\quad \rho=1,\]
\[a=\varepsilon^{pr}(\mu,t_{k}),\quad b=\|e^{pr}(\mu,t_k)\|,\quad \rho=\alpha(\mu),\]
and by Equation \eqref{eq2.2}, we obtain
\begin{align*}
&a(e^{pr}(\mu,t_k),e^{pr}(\mu,t_k)+c(e^{pr}(\mu,t_k),e^{pr}(\mu,t_k)\leq 
c\sum_{m=1}^{k-1}(e^{pr}(\mu,t_m),e^{pr}(\mu,t_m)+\frac{1}{\alpha(\mu)} (\varepsilon^{pr}(\mu,t_{k})^2,
\end{align*}
\begin{align*}
&a(e^{pr}(\mu,t_k),e^{pr}(\mu,t_k)+c\|e^{pr}(\mu,t_k)\|^2\leq 
c\sum_{m=1}^{k-1}\|e^{pr}(\mu,t_m)\|^2+\frac{1}{\alpha(\mu)} (\varepsilon^{pr}(\mu,t_{k})^2.
\end{align*}
We now perform it from $k^{'}=1$ to $k$ and $e^{pr}(\mu,t_0)=0$, which completes the proof.
For $k'^{'}=1$,
\begin{align}\label{eq7.a}
&a(e^{pr}(\mu,t_1),e^{pr}(\mu,t_1))+c\|e^{pr}(\mu,t_1)\|^2\leq \frac{1}{\alpha(\mu)} (\varepsilon^{pr}(\mu,t_{k})^2=\Delta_1 \Rightarrow \nonumber\\& c\|e^{pr}(\mu,t_1)\|^2\leq \Delta_1,
\end{align}
 $k^{'}=2$,
\begin{align*}
&a(e^{pr}(\mu,t_2),e^{pr}(\mu,t_2))+c\|e^{pr}(\mu,t_2)\|^2\leq 
c\|e^{pr}(\mu,t_1)\|^2+\frac{1}{\alpha(\mu)} (\varepsilon^{pr}(\mu,t_{2})^2.
\end{align*}
By equation \eqref{eq7.a}, we obtain
\begin{align*}
&a(e^{pr}(\mu,t_2),e^{pr}(\mu,t_2))+c\|e^{pr}(\mu,t_2)\|^2 \leq \Delta_1+\frac{1}{\alpha(\mu)} (\varepsilon^{pr}(\mu,t_{2})^2=\Delta_2,\Rightarrow\nonumber\\& c\|e^{pr}(\mu,t_2)\|^2\leq \Delta_2,\\
&\vdots\\
k^{'}=k,\\
&a(e^{pr}(\mu,t_k),e^{pr}(\mu,t_k)+c\|e^{pr}(\mu,t_k)\|^2\leq [ \frac{1}{\alpha(\mu)}(\varepsilon^{pr}(\mu,t_{k}))^2+\sum_{k^{'}=1}^{k-1}\Delta_{k^{'}}]\Rightarrow\\
&\|e^{pr}(\mu,t_k)\leq [ \frac{1}{\alpha(\mu)}(\varepsilon^{pr}(\mu,t_{k}))^2+\sum_{k^{'}=1}^{k-1}\Delta_{k^{'}}]^{\frac{1}{2}}.
\end{align*}
\end{proof}
\begin{lemma}
Let $e^{du}(\mu,T-t_k)\equiv \bar{p}_(\mu,T-t_k)-\bar{p}_N(\mu,T-t_k)$ be the error in the dual variable and define the norm 
\[\|v(\mu,t_k)\|^{du}=((v(\mu,t_k),v(\mu,t_k);\mu)+a(v(\mu,t_k),v(\mu,t_k);\mu))^{\frac{1}{2}}.\]
The error in the dual variable is then bounded by
\[\|e^{du}(\mu,T-t_k)\|^{du}\leq [\frac{1}{\alpha(\mu)}(\varepsilon^{du}(\mu,{T-t_{k}} ))^2+\sum_{k^{'}=k+1}^K\Delta_{k^{'}}]^{\frac{1}{2}}\]
where 
\[\Delta_{k^{'}}=\sum_{k^{"}=k^{'}}^{K}\frac{1}{\alpha(\mu)}(\varepsilon^{du}(\mu,{T-t_{k^{"}}}))^2.\]
\end{lemma}
\begin{proof}
Follow proof of Lemma \ref{Le7.1}. For more detail, see \cite{kar}
\end{proof}
\section{Numerical simulations}\label{se8}
In this section, numerical experiments are conducted to validate the proposed method. We have solved the following three fractional PDE constrained optimization problems which are similar to test cases in \citep{19}. The simulation
is conducted on Matlab 7.
\begin{example}\label{ex7.1}
 In the problem $ (P) $, we take $T=1$, $x\in [0,1]$, $\alpha=0.7$ and
\begin{align*}
&y_d=\nu(2(t-1)^3x(x-1)+12t(t-1)^2x(x-1)+3t^2(2t-2)x(x-1))+\\&
+t^2(1-t)^3x(x-1).
\end{align*}

The graphs of finite element solution and reduced solution of $ y(x,t) $ and $ p(x,t) $ for $ t=0.2,\,0.4,\,0.6,\,0.8 $ with $ N=7 $,  $ \gamma =10^{-6} $ and $\mu=0.75$ are  plotted in Fig. \ref{fig1}. As you can see in Fig. \ref{fig1} the finite element solution and reduced solution are extremely close.  In Fig. \ref{fig2}, the error functions $ \|y_h-y_N \|_2$ and $ \|p_h-p_N \|_2$ with $ \nu =10^{-6} $ are depicted.

\begin{figure}[h!]\label{fig1}
\centering
\subfloat[\label{fig1a}]{\includegraphics[scale=.2]{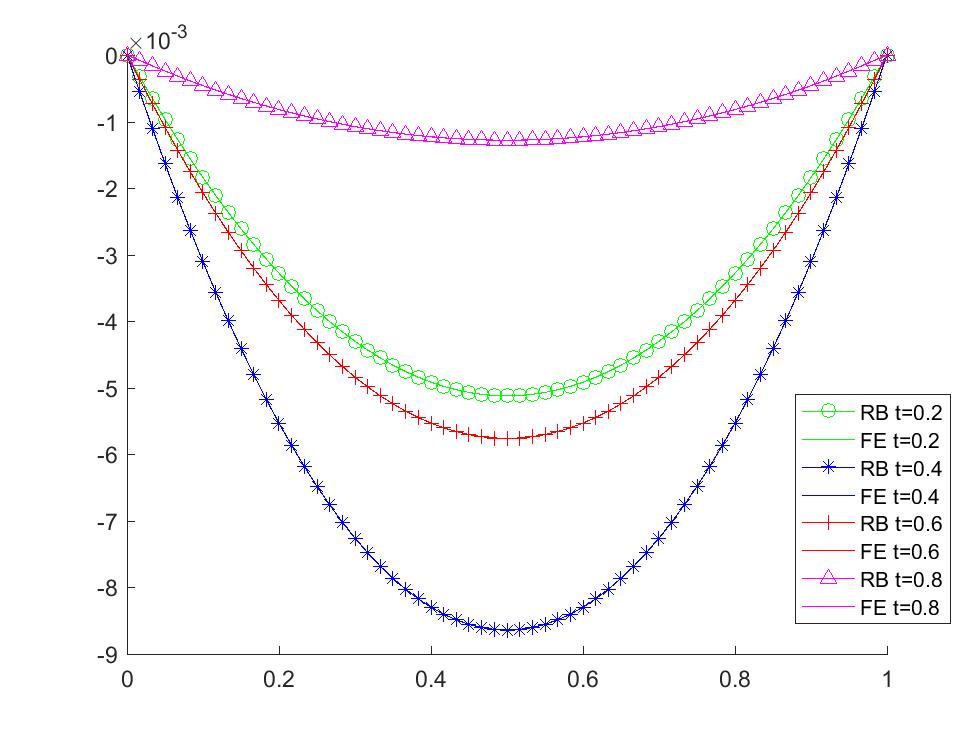}}
\subfloat[\label{fig1b}]{\includegraphics[scale=.2]{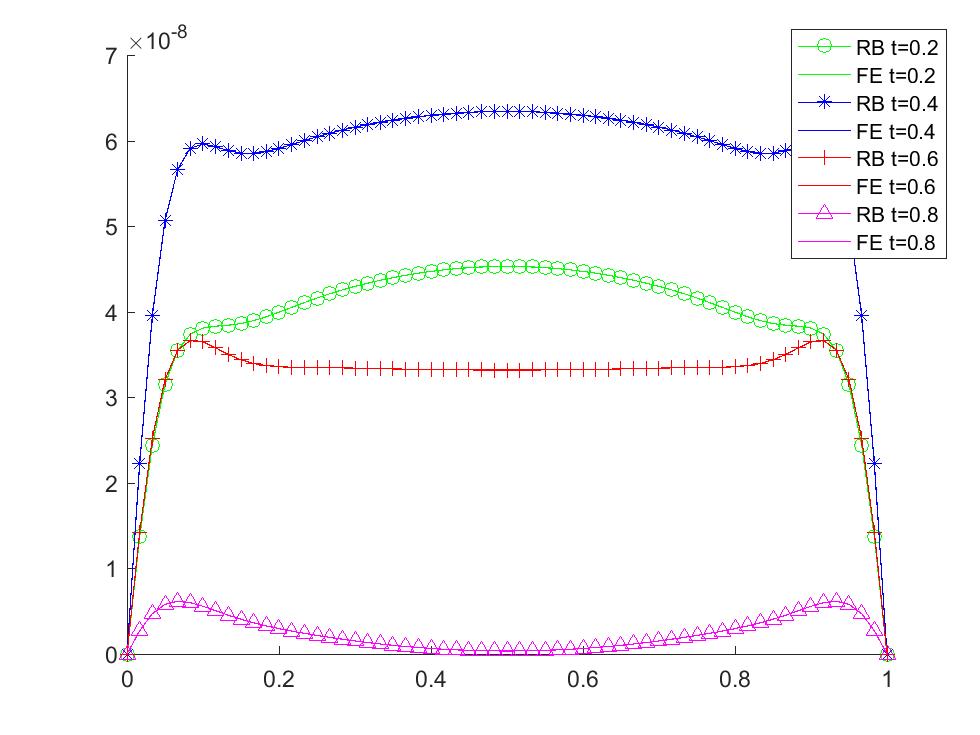}}
\caption{Comparisons between finite element and reduced  solutions of $y(x,t)$ \eqref{fig1a} and  $ p(x,t) $ \eqref{fig1b} in $t=0.2$s, $t=0.4$s, $t=0.6$s, $t=0.8$s with $ \gamma =10^{-6}, \mu=0.75 $ in Example \ref{ex7.1}.}\label{fig1}
\end{figure}

\begin{figure}[h!]\label{fig2}
\centering
\subfloat{\includegraphics[scale=.2]{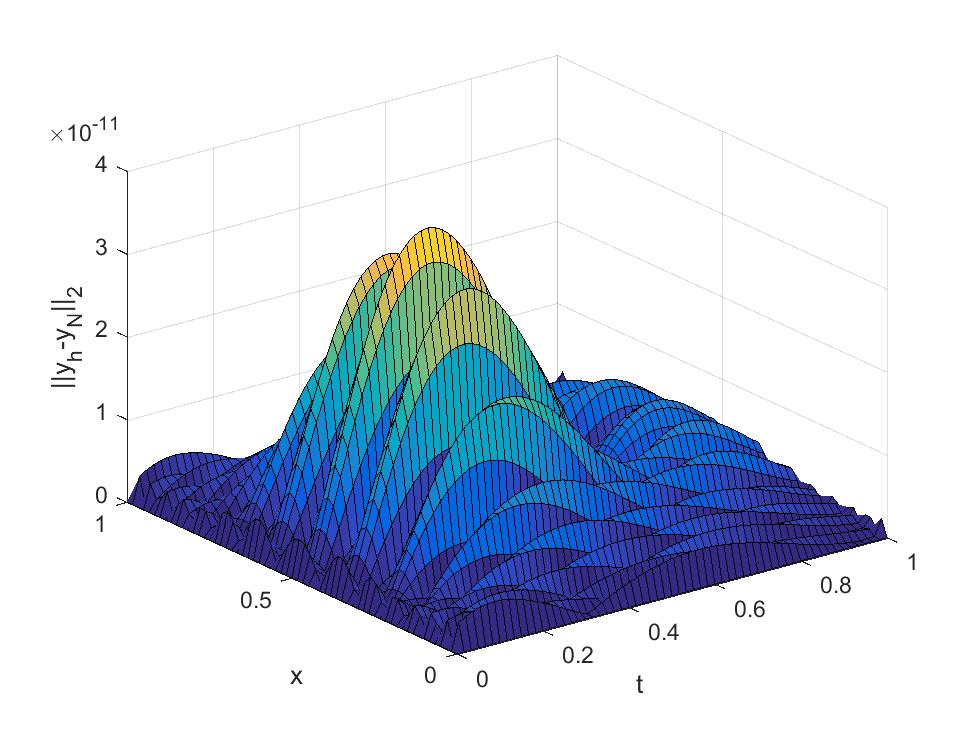}}
\subfloat{\includegraphics[scale=.2]{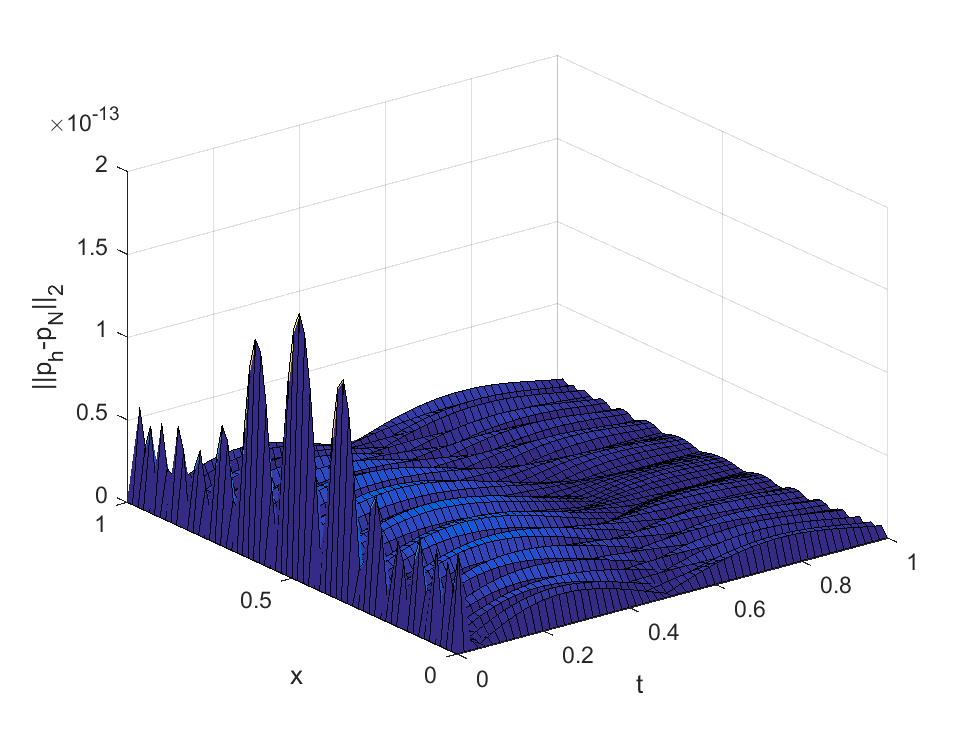}}
\caption{\footnotesize{Plots of $ \|y_h-y_N\| _2$ and $ \|p_h-p_N\|_2 $  with $ \gamma =10^{-6}, \mu=0.75 $ in Example \ref{ex7.1}.}}\label{fig2}
\end{figure}

See Fig. \ref{fig3} for the error bound of the reduced basis method for the state variable and adjoint variable. This method choose 7 parameters to construct the state space and adjoint space.

\begin{figure}[h!]\label{fig3}
\centering
\subfloat[\label{fig3a}]{\includegraphics[scale=.2]{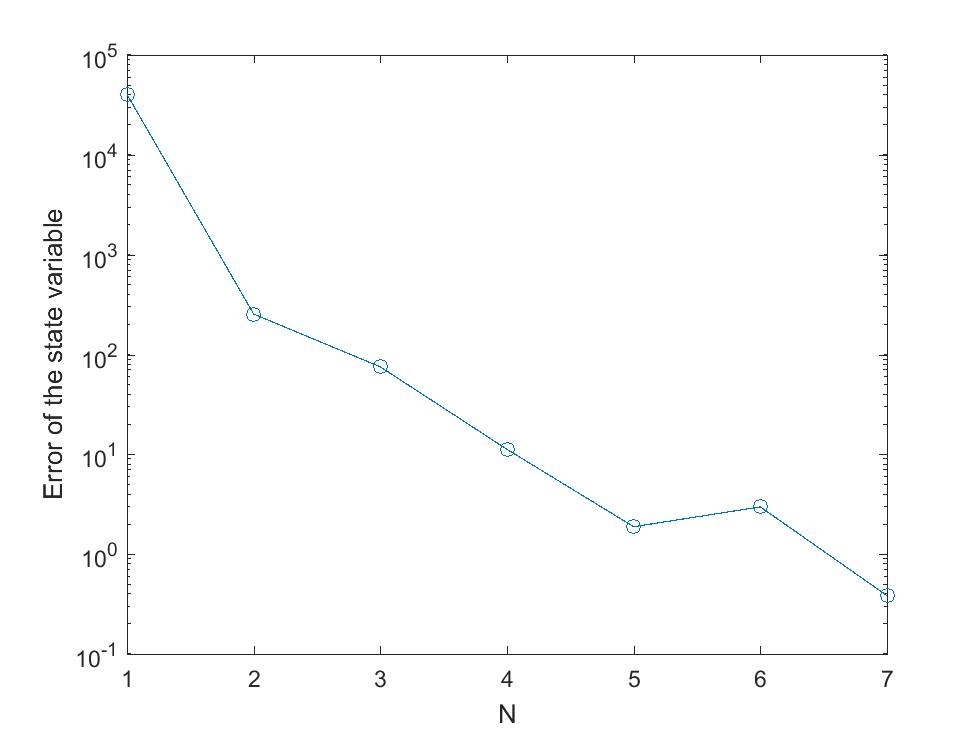}}
\subfloat[\label{fig3b}]{\includegraphics[scale=.2]{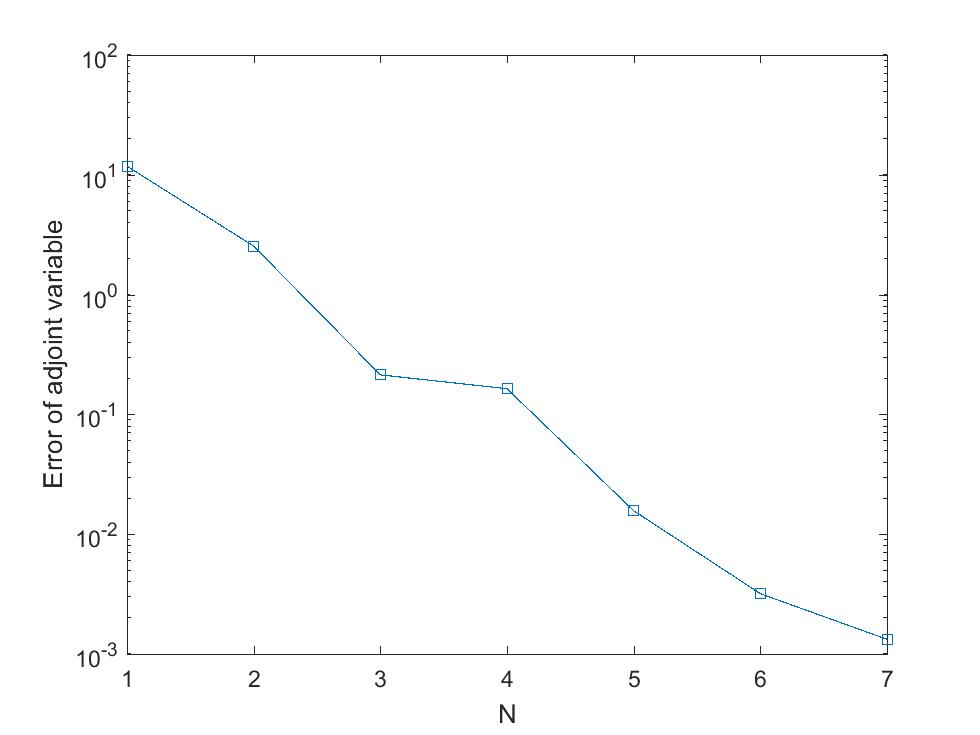}}
\caption{\footnotesize{The error bound for the state variable \eqref{fig3a} and adjoint variable  \eqref{fig3b}. }}
\end{figure}
\end{example}

\begin{example}\label{ex7.2}
In problem $ (P) $, we take $T=1$, $x\in [0,1]$, $\alpha=0.99$ and
\begin{align*}
&y_d=\gamma(2\pi ^2t(t-1)^2(t-2)^2+\pi ^4 t^2(t-1)^2(t-2)^2-2t^2(t-1)^2-2t^2(t-2)^2\\&-2(t-1)^2(t-2)^2-2t^2(2t-2)(2t-4)-4t(2t-2)(t-2)^2-4t(2t-4)(t-1)^2-\\&2\pi ^2t(t-1)^2(t-2)^2)\sin(\pi x)+t^2(1-t)^2(2-t)^2\sin(\pi x).
\end{align*}
The graphs of finite element  and reduced solutions of $ y(x,t) $ and $ p(x,t) $ for $ t=0.2, 0.4, 0.6, 0.8 $ with $ N=1 $ and $ \gamma =10^{-8} $ are plotted in Fig. \ref{fig4}. In Fig. \ref{fig5}, the error functions $ \|y_h-y_N\| $ and $ \|p_h-p_N\| $ with $ \gamma =10^{-8} $ are shown.

\begin{figure}[h!]
\centering
\subfloat[\label{fig4a}]{\includegraphics[scale=.2]{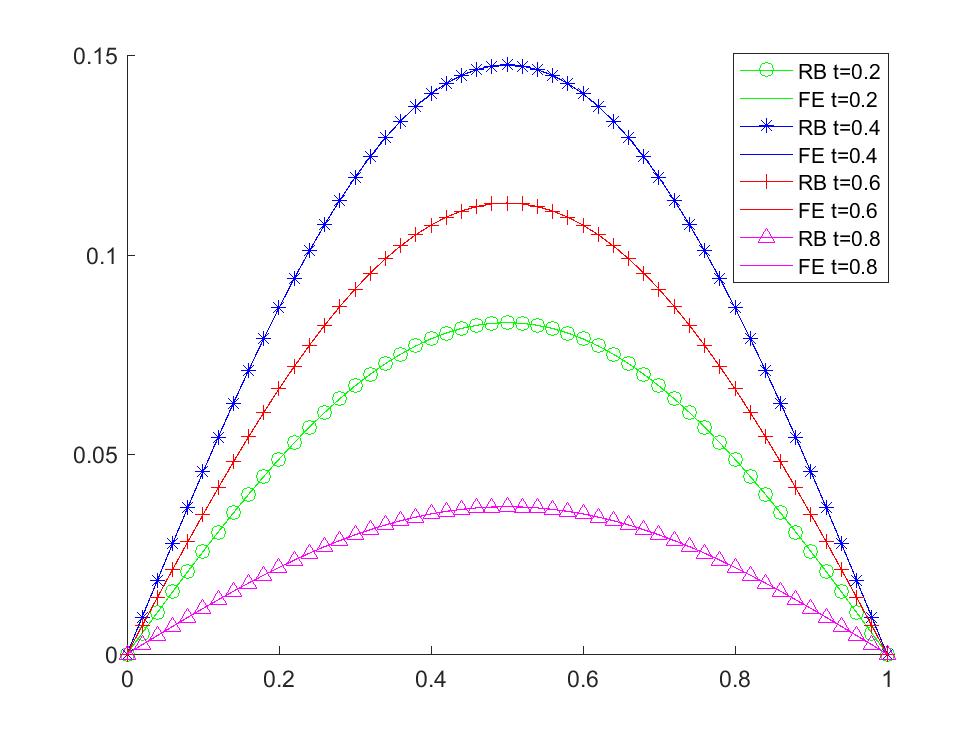}}
\subfloat[\label{fig4b}]{\includegraphics[scale=.2]{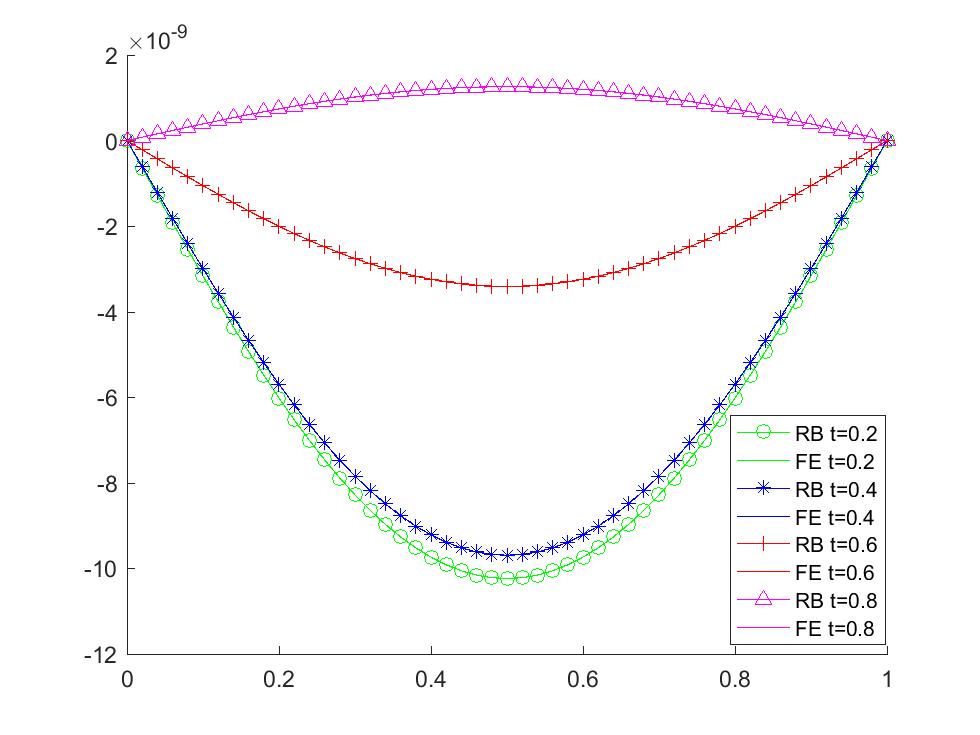}}
\caption{Comparisons between finite element and reduced  solutions of $y(x,t)$ \eqref{fig4a} and  $ p(x,t) $ \eqref{fig4b} in $t=0.2$s, $t=0.4$s, $t=0.6$s, $t=0.8$s with $ \gamma =10^{-8}, \mu=0.6 $ in Example \ref{ex7.2}.}\label{fig4}
\end{figure}
\clearpage
\begin{figure}[h!]
\centering
\subfloat{\includegraphics[scale=.2]{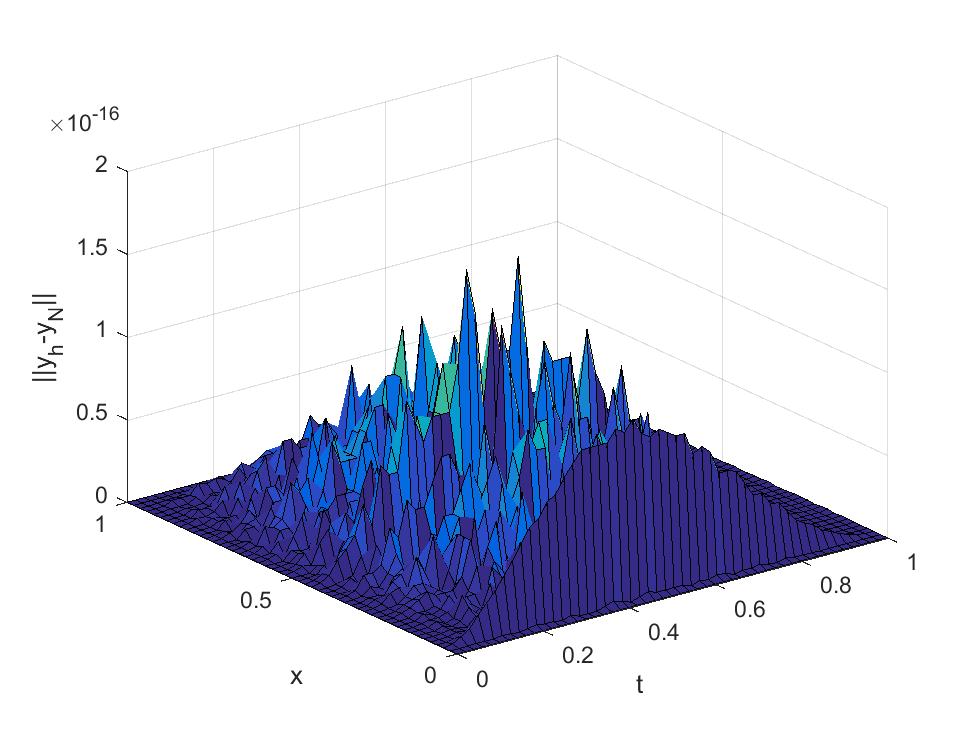}}
\subfloat{\includegraphics[scale=.2]{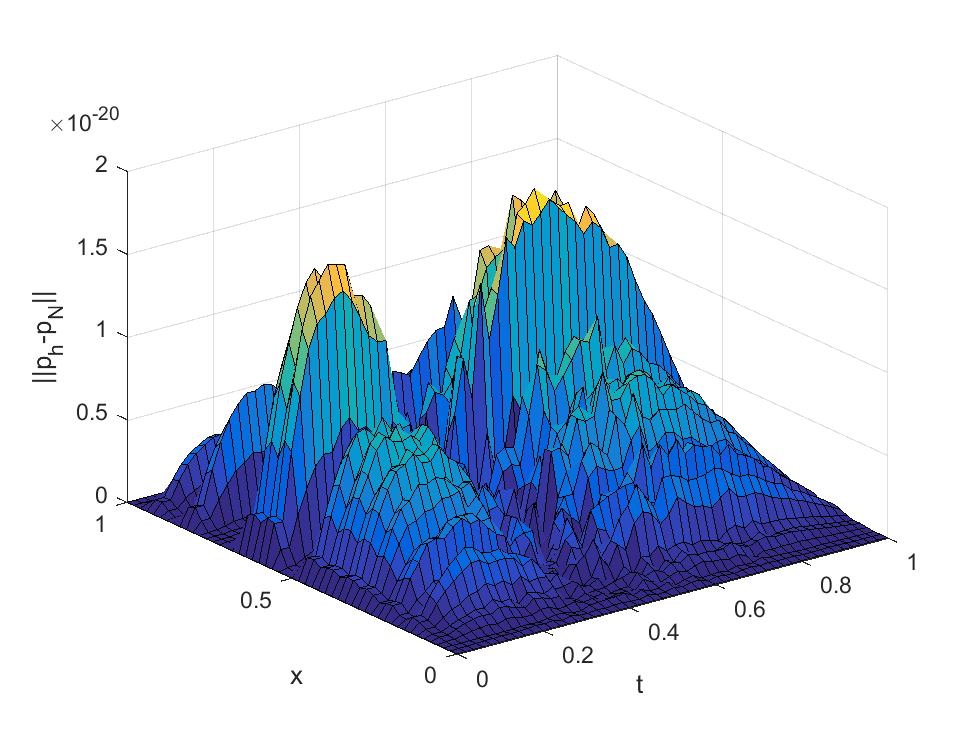}}
\caption{\footnotesize{Plots of $ \|y_h-y_N\| _2$ and $ \|p_h-p_N\|_2 $  with $ \gamma =10^{-8}, \mu=0.6 $ in Example \ref{ex7.2}.}}\label{fig5}
\end{figure}
In this example, the number of chosen parameter is one, so we did not plot the error bound.
\end{example}

\begin{example}\label{ex7.3}
In the problem $ (P) $, we take $T=1$, $x\in [0,1]$, $\alpha=0.7$ and
\begin{align*}
&y_d=\gamma((16\pi ^4t^3(t-1)^3-3t^3(2t-2)-18t^2(t-1)^2-6t(t-1)^3)\cos(2\pi x))+\\
&\gamma(3t^3(2t-2)+18t^2(t-1)^2+6t(t-1)^3)+t^3(1-t)^3(1-\cos(2\pi x)).
\end{align*}
The graphs of finite element  and reduced solutions of $ y(x,t) $ and $ p(x,t) $ for $ t=0.2, 0.3, 0.5, 0.8 $ with $ N=5$ and $ \gamma =10^{-7} $ are plotted in Fig. \ref{fig6}. In Fig. \ref{fig7}, the error functions $ \|y_h-y_N\| $ and $ \|p_h-p_N\| $ with $ \gamma =10^{-7} $ plotted.
\begin{figure}[h!]
\centering
\subfloat[\label{fig6a}]{\includegraphics[scale=.2]{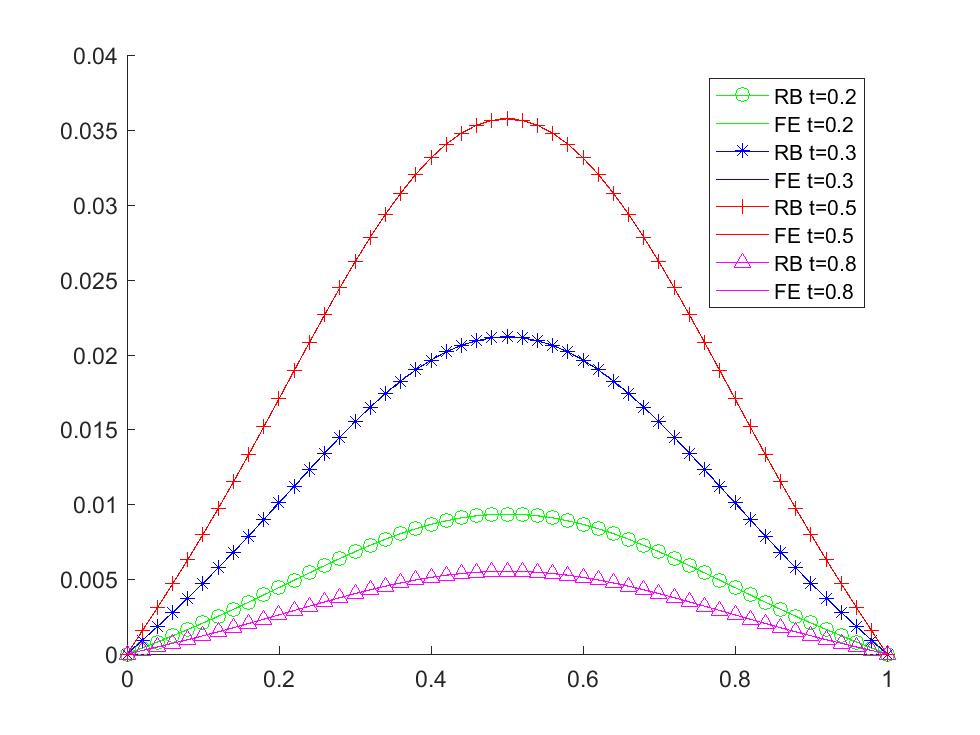}}
\subfloat[\label{fig6b}]{\includegraphics[scale=.2]{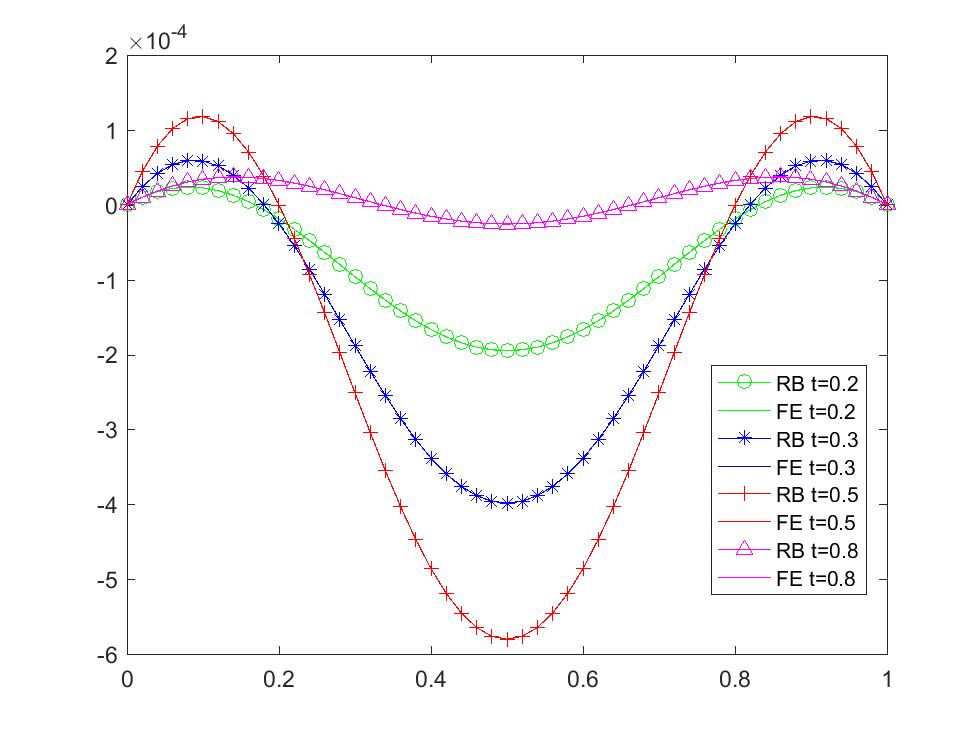}}
\caption{Comparisons between finite element and reduced  solutions of $y(x,t)$ \ref{fig4a} and  $ p(x,t) $ \ref{fig4b} in $t=0.2$s, $t=0.3$s, $t=0.5$s, $t=0.8$s with $ \gamma =10^{-3}, \mu=0.45 $ in Example \ref{ex7.3}.}\label{fig6}

\end{figure}
\begin{figure}[h!]
\centering
\subfloat{\includegraphics[scale=.2]{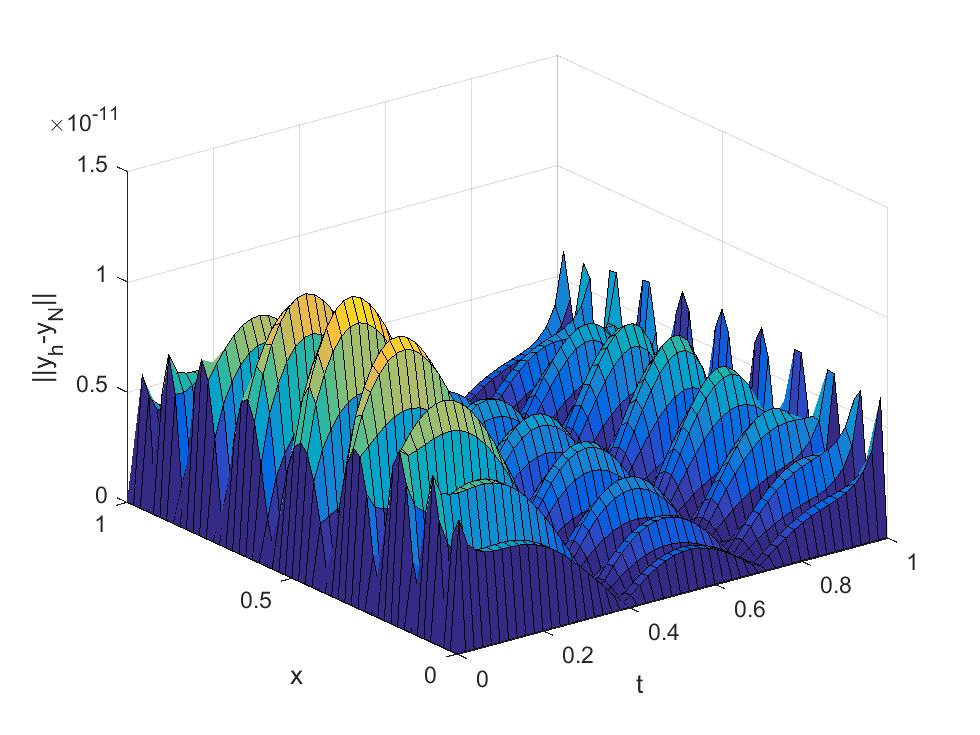}}
\subfloat{\includegraphics[scale=.2]{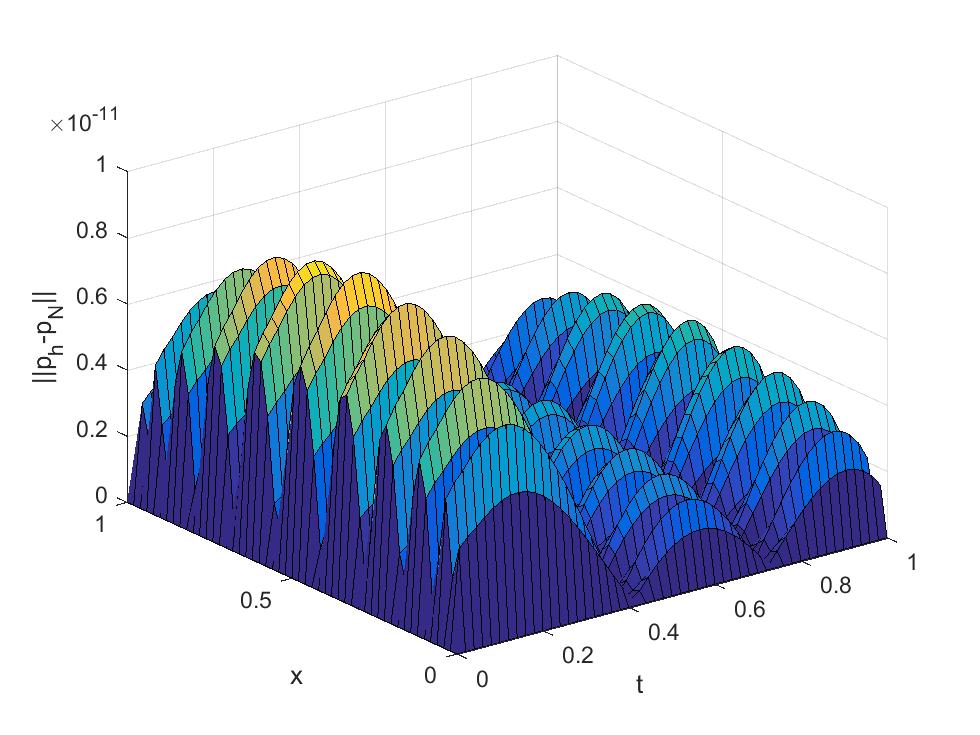}}
\caption{\footnotesize{Plots of $ \|y_h-y_N\|_2 $ and $ \|p_h-p_N\|_2 $  with $ \gamma=10^{-3} $ in Example \ref{ex7.3}.}}\label{fig7}
\end{figure}
See Fig. \ref{fig8} for the error bound of the reduced basis method for the state variable and adjoint variable. This method choose 8 parameters to construct the state space and adjoint space.

\begin{figure}[h!]
\centering
\subfloat[\label{fig8a}]{\includegraphics[scale=.2]{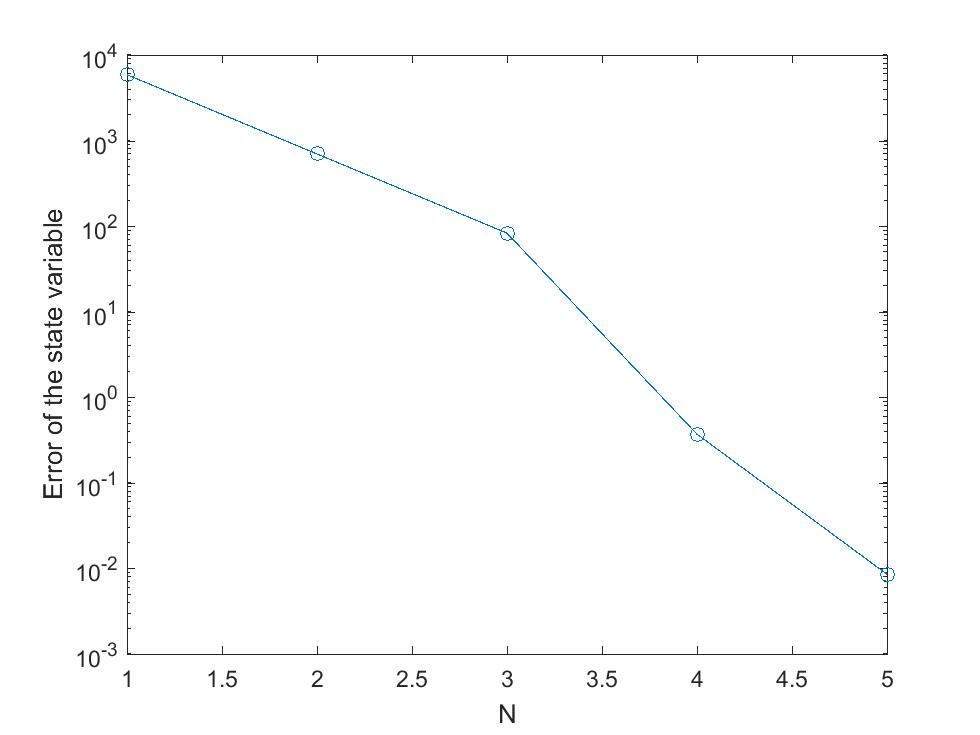}}
\subfloat[\label{fig8b}]{\includegraphics[scale=.2]{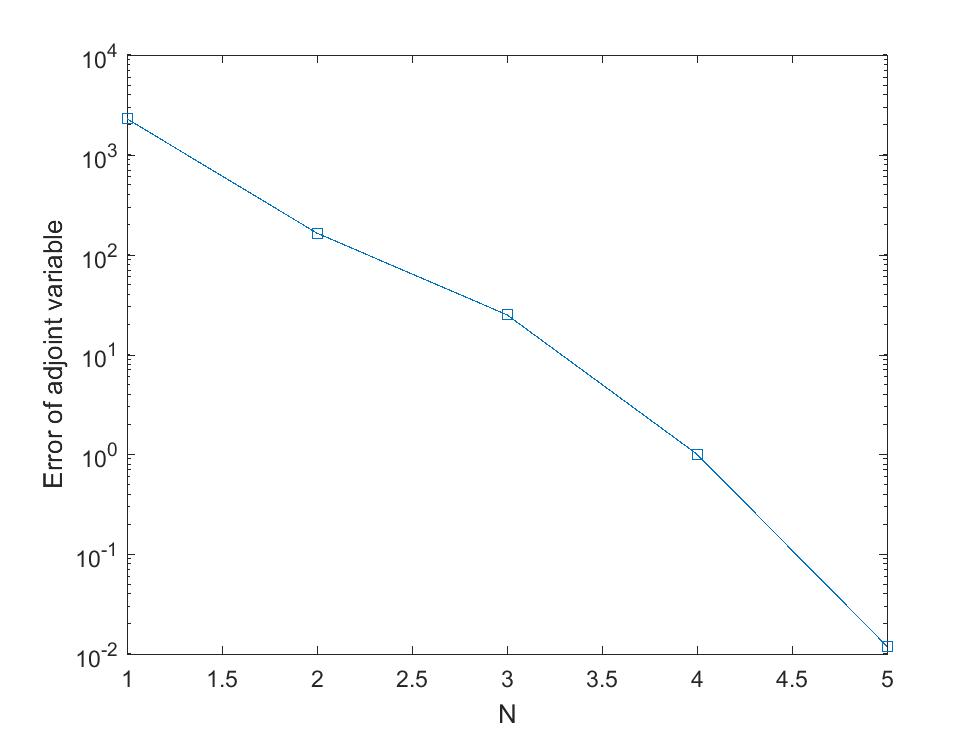}}
\caption{\footnotesize{The error bound for the state variable \eqref{fig8a} and adjoint variable  \eqref{fig8b}. }}\label{fig8}
\end{figure}
\end{example}

\section{Conclusion}
In this paper, we use parametric model order reduction using reduced basis methods as an effective tool for obtaining a quick solution of fractional PDE constrained optimization problem. The used technique is applied to solve three test problems and the resulting solutions are in good agreement with the known exact solutions. For the sake of simplicity, we only considered the one-dimensional case with standard initial and boundary conditions, but the method can be extended to multi-dimensional cases with even non-classic boundary conditions which is the subject of the authors. The accuracy of numerical solution by this method is much higher than the classical numerical solutions. Numerical solutions are obtained efficiently and the stability is maintained for randomly perturbed data.



\end{document}